\documentclass{article}
\usepackage{amsfonts,amsmath,amscd,amsthm,graphicx}

\newtheorem{prop}{Proposition}

\newtheorem{lem}{Lemma}

\def \ee {\begin{equation}}
\def \eee {\end{equation}}

\newcommand{\Sy}{\mathrm{S^+(p,n)}}

\newcommand{\RR}{{\mathbb R}}

\newcommand{\norm}[1]{\lVert#1\rVert}

\newcommand{\tr}[1]{\text{Tr}\left(#1\right)}

\usepackage{xcolor}

\begin{document}

\title{Rank-preserving geometric means of positive semi-definite matrices}

\author{Silv\`{e}re Bonnabel\thanks{S. Bonnabel is with Robotics center, Math\'{e}matiques et Syst\`{e}mes,
    Mines ParisTech, Boulevard Saint-Michel 60, 75272 Paris, France, silvere.bonnabel@mines-paristech.fr.}, Anne Collard, Rodolphe Sepulchre \thanks{A. Collard and R. Sepulchre are with Departement of Electrical Engineering and Computer Science,
    University of Li\`{e}ge, 4000 Li\`{e}ge, Belgium, anne.collard@ulg.ac.be, r.sepulchre@ulg.ac.be.}}

  \maketitle

\begin{abstract}
{The generalization of the geometric mean of positive scalars to positive definite matrices has attracted considerable attention since the seminal work of Ando. The paper generalizes this framework of matrix means by proposing the definition of a rank-preserving mean for two or an arbitrary number of positive semi-definite matrices of fixed rank. The proposed mean is shown to be geometric in that it satisfies all the expected properties of a rank-preserving geometric mean.   The work is motivated by operations on low-rank approximations of positive definite matrices in high-dimensional spaces. }
  \end{abstract}


\section{Introduction}
Positive definite matrices have become fundamental
computational objects in
many areas of  engineering and applied mathematics. They appear as covariance
matrices in statistics, as variables in convex and
semidefinite programming, as unknowns of important matrix (in)equalities  in systems and control theory, as kernels in machine learning, and as diffusion tensors in medical imaging, to cite a few. These applications have motivated the development of  differential calculus over positive definite matrices.  As a most basic operation, this calculus requires the proper definition of a mean.  In particular, much research has been devoted to the matrix generalization of  the geometric mean $\sqrt{ab}$ of two positive numbers $a$ and $b$  (see for instance Chapter 4 in \cite{bhatia06} for an expository and insightful treatment of the subject).  The further extension of a geometric mean from two to an arbitrary number of positive definite matrices is an active current research area \cite{ando, petz,bhatia-2006,LL2011,BK2011,holbrook,Bini,lim2}.
It has been increasingly recognized that from a theoretical point of view \cite{faraut} as well as in numerous  applications \cite{pennec-06,fletcher,moakher06,moakher05,ando,arsigny,petz,burbea-rao,skovgaard,smith-2005}, matrix geometric means  are to be preferred to their arithmetic counterparts for developing  a calculus in the cone of positive definite matrices.

The fundamental and axiomatic approach of Ando \cite{ando} (see also \cite{petz}) reserves the adjective ``geometric"  to a definition of mean that enjoys (at least) all the following properties:
\begin{itemize}
\item[(P1)] Consistency with scalars. If $A$ { and} $B$ commute, {then} $M(A,B)=(AB)^{1/2}$.
\item [(P2)] Joint homogeneity. For $\alpha,\beta>0$ $$M(\alpha A,\beta B)=(\alpha\beta)^{1/2}M(A,B).$$
\item [(P3)] Permutation invariance. $M(A,B)=M(B,A).$
\item [(P4)] Monotonicity. If $A\leq A_0$ (i.e. ($A_0-A$) is a positive matrix) and $B\leq B_0$, the means are comparable and verify $M(A,B)\leq M(A_0, B_0)$.
\item [(P5)] Continuity from above. If $\{A_n\}$ and $\{B_n\}$ are monotonic decreasing sequence (in the Lowner matrix ordering) converging to $A$, $B$ then we have $\lim M(A_n, B_n)=M(A,B)$.
\item [(P6)] Congruence invariance. For any $G\in \mathrm{Gl(n)}$, we have $M(GAG^T , GBG^T)=GM(A,B)G^T$.
\item [(P7)] Self-duality. $M(A,B)^{-1}  =M(A^{-1},B^{-1})$.
\end{itemize}

The present paper seeks to  extend  geometric means defined on  the open cone  $\mathrm{P_n}$ to the the set of positive semi-definite matrices of fixed rank $p$, denoted by $\Sy$, which lies on the boundary of $\mathrm{P_n}$. Our motivation is primarily computational: with the growing use of low-rank approximations of matrices as a way to retain  tractability in large-scale applications, there is a need to extend the calculus of positive definite matrices to their low-rank counterparts.  The classical approach in the literature is to extend the definition of a mean from the interior of the cone to the boundary of the cone by a continuity argument. As a consequence, this topic has not received much attention. This approach has however serious limitations from a computational viewpoint because it is not rank-preserving. For instance Ando's geometric mean of two semi-definite matrices having rank $p<n/2$  is almost surely null {with this definition}.

We depart from this approach by viewing a rank $p$ positive semi-definite matrix as the projection of a positive definite matrix in a $p$-dimensional subspace. Our proposed mean lies in the mean subspace, a well-defined and classical concept. The proposed mean is rank-preserving, and it possesses all the properties listed above, {modulo a few adaptations imposed by a rank-preserving concept}:  (P1) is impossible to retain {unless}  the rank of $AB$ is {equal to} the rank of $A$ and $B$. Indeed, as  the mean must preserve the rank, it can not be equal to $(AB)^{1/2}$ if the latter condition is not satisfied. In turn, as $A$ and $B$ are supposed to commute it means that they  must be supported by the same subspace. {Also} (P6) will be shown to be impossible to retain when the rank is preserved. Indeed it is this property {that causes} Ando's geometric mean of two matrices of sufficiently small rank {to be} almost surely null.  In (P7) inversion must obviously  be replaced with pseudo-inversion. Letting $A\circ B$ denote the desired  mean in $\Sy$, we suggest to replace those three properties  with:
\begin{itemize}
\item[(P1')] Consistency with scalars. If $A,B$ commute and are supported by the same subspace, $A\circ B=(AB)^{1/2}$.
\item[(P6')] Invariance to scalings and rotations. For $(\mu,P)\in \RR^*_+ \times \mathrm{O(n)}$ we have $(\mu P^TA\mu P)\circ (\mu P^TB\mu P)=\mu P^T(A\circ B )\mu P$.
\item[(P7')] Self-duality. $(A\circ B)^{\dag}  =A^{\dag}\circ B^{\dag}$, where $\dag$ is the pseudo-inversion.
\end{itemize}

{In the recent work \cite{bonnabel-sepulchre-simax}, we used a
Riemannian framework to introduce a  geometric mean of two matrices
in $\Sy$ that was shown to satisfy those properties. The present
paper further develops the concept by providing an intuitive
characterization and a closed formula for its calculation.
Furthermore, we show that the concept extends to the definition of a
geometric mean for an arbitrary number of matrices, thereby
providing a low-rank counterpart of  recent work on positive
definite matrices  \cite{ando,petz,bhatia-2006,holbrook,Bini}.

The structure of the paper is as follows.  Sections 2 and 3 are
mainly  expository.  In Section 2, we review the theory of Ando in
the cone of positive definite matrices and we illustrate the
shortcomings of the continuity argument for a rank-preserving mean
to be defined on the boundary of the cone. In Section 3,   {we
review the Riemannian interpretation of Ando's mean of two matrices
$A$ and $B$ as the midpoint of the geodesic joining $A$ and $B$ for
the affine invariant metric of the cone and introduce the Riemannian
concept of Karcher mean. Section 4 develops the proposed  geometric
mean for an arbitrary number of matrices in  the set $\Sy$. The
geometric  properties of this mean are characterized in Section 5.
Finally, Section 6 illustrates the relevance of a rank-preserving
mean in the context of filtering.} {Preliminary results can be found
in \cite{Bonnabel:MTNS10}}.

\subsection{ Notation}
\begin{itemize}
\item $\mathrm{P_n}$ is the set of symmetric positive definite $n\times n$ matrices.
\item $\mathrm{S^+(p,n)}$ is the set of symmetric positive semidefinite $n\times n$ matrices of rank $p\leq n$.
\item $\mathrm{Sym(n)}$ is the vector space of symmetric $n\times n$  matrices.

\item $\mathrm{St(p,n)}=\mathrm{O(n)}/\mathrm{O(n-p)}$ is the  Stiefel manifold, that is, the set of $n\times p$
matrices with orthonormal columns: $U^TU=I_p$.
\item $\mathrm{Gr(p,n)}$ is  the Grassmann manifold, that is,  the set of $p$-dimensional subspaces of $\RR^n$.
It can be represented by the equivalence classes $\mathrm{St(p,n)}/\mathrm{O(p)}$.
\item span($A$) is the subspace of $\RR^n$ spanned by the columns of $A\in \RR^{n\times n}$.
\end{itemize}


\section{An analytic viewpoint: Ando's approach}

\subsection{Mean of two matrices $A_1$ and $A_2$}
For positive scalars, the homogeneity property (P2) implies
$M(a_1,a_2)=a_1M(1,a_2/a_1)=a_1f(a_2/a_1)$ with $f$ a monotone
increasing continuous function. In a non-commutative matrix setting,
one can write \ee \label{geometricmean}
M(A_1,A_2)=A_1^{1/2}f(A_1^{-1/2}A_2A_1^{-1/2})A_1^{1/2}\, . \eee
with $f$ a matrix monotone increasing  function. Several  geometric
means can be defined this way (see e.g. \cite{arsigny}). The
well-established geometric mean  of two full-rank matrices
popularized by Ando \cite{ando78,pusz,ando} corresponds to the case
$f(X)=X^{1/2}$, generalizing the scalar geometric mean. It is
defined by \ee \label{Ando}
A_1\#A_2=A_1^{1/2}(A_1^{-1/2}A_2A_1^{-1/2})^{1/2}A_1^{1/2}. \eee It
satisfies all the propositions (P1-P7) listed above. There are many
equivalent definitions of the Ando geometric mean in the literature.

A geometric mean satisfying (\ref{geometricmean}) is defined for positive definite matrices, that is, for elements in the {open} cone of positive definite matrices. Rank-deficient matrices lie on the {closure} of the cone. As a consequence, the natural idea to extend a geometric mean on the boundary is to use a continuity argument. The resulting mean satisfies all the properties above (except for (P7) that must be formulated using pseudo-inversion), but it is not rank-preserving. Indeed, let  $A_1=\text{diag(}4,\epsilon^2)$ and $A_2=\text{diag}(\epsilon^2,1)$ where  $\epsilon\ll 1$. These two matrices belong to $P_2$, and their  geometric  mean is $A_1\#A_2=$diag($2\epsilon $, $\epsilon$). In the limit (rank-deficient) situation $\epsilon\rightarrow 0$, the mean becomes the null matrix diag($0,0)$. The following proposition {shows that this example is not pathological.}
\begin{prop}\label{pp1}
If (P6) is satisfied, the geometric mean of two matrices of $\Sy$ is almost surely null if $p<n/2$.
\end{prop}
\begin{proof} In \cite{ando}, it is  proved (Theorem 3.3) that (P6) implies that  the range of the geometric mean of $A_1$ and $A_2$ is the intersection of the subspaces  $\mathrm{span}(A_1)$ and $\mathrm{span}(A_2)$ (this can be proved letting a sequence of matrices of $\mathrm{Gl}(n)$ converge to the orthoprojector on Ker $A_1$).   Since the intersection of two random subspaces of dimension $p$ is almost surely empty as long as $n-p>p$, the range of $\mathrm{span}(A_1)\cap \mathrm{span}(A_2)$ is almost surely the null space, which proves the claim.\end{proof}

A rank-preserving mean thus requires a different approach. We seek to retain most of the properties (P1-P7), but we will see that  three of them  must be relaxed  to define a rank-preserving mean. The major relaxation consists in  choosing a smaller invariance group in (P6), replacing the general linear group   $\mathrm{Gl(n)}$ with the smaller but meaningful  group of scalings and rotations $\RR_+^* \times O(n)$.

\subsection{Mean of an arbitrary number of matrices $A_1,\cdots, A_N$}\label{multmean:sec}
A geometric mean of an arbitrary number of matrices, that extends the geometric mean of two matrices \eqref{Ando} is not very well-established. Indeed the definitions based on equations \eqref{geometricmean} for instance, are not easily generalized.  Several possible definitions have appeared in the literature and we shall not review all of them. In any case, it seems natural to reserve the {adjective}   geometric, to a mean that satisfies the following properties (PP1-PP7). They are a natural  extension of (P1-P7) to the case of three matrices, and the extension to an arbitrary number of matrices is straightforward.  (PP1) if A,B,C commute $M(A,B,C)=(ABC)^{1/3}$.  (PP2)  $M(\alpha A,\beta B,\gamma C)=(\alpha\beta\gamma)^{1/3} M(
A,B,C). $
 (PP3)  $M(A,B,C)=M(\pi(A,B,C))$ for any permutation $\pi$.
 (PP4)  The map $(A,B,C)\mapsto M(A,B,C)$ is monotone.
 (PP5)  If $\{A_n\}$, $\{B_n\}$, $\{C_n\}$ are monotonic decreasing sequences converging to $A,B,C$ then $\lim M(A_n , B_n,C_n)=M(A,B,C)$.
 (PP6) For any $G\in \mathrm{Gl(n)}$ we have $M(GAG^T , GBG^T,GCG^T)=GM(A,B,C)G^T$.
 (PP7)  $M(A,B,C)^{-1}  =M(A^{-1},B^{-1},C^{-1})$.

This axiomatic approach has been proposed {in} \cite{ando}, and  the authors  have defined a mean we shall denote $\mathrm{alm}(A_1,\cdots, A_n)$, adopting the notation of \cite{bhatia-2006}. This mean  is defined as the common limit {of} a converging sequence of  matrices, and it was proved to preserve properties (P1-P7) as well as their extension (PP1-PP7) to three or more matrices. Other computationally more efficient geometric means having the desired properties have since been proposed (see e.g. \cite{Bini}, and \cite{lim2} for a weighted geometric mean).


\section{A geometric viewpoint: geometric mean as a Riemannian mean}

{Ando's mean (\ref{Ando}) has the alternative Riemannian interpretation of the  midpoint of a geodesic connecting the matrices $A$ and $B$. This connection appears for instance in \cite{corach}.  Because this Riemannian interpretation is at the root of our proposed rank-preserving mean, it is reviewed in this section.}

\subsection{Riemannian mean and Karcher mean on a Riemannian manifold}

The arithmetic mean of $N$ positive  numbers in $\RR_+^*$ is
defined as $M(x_1, \cdots,x_N)=\frac{1}{n}\sum_1^nx_i$. It has the
variational property of being the unique minimizer of the sum of
squared distances $M(x_1, \cdots,x_N)=\text{argmin}_x\sum_i
d(x,x_i)^2$ where $d$ is the Euclidean distance in $\RR$. In the
same way, the geometric mean of $n$ positive scalars minimizes the
same sum {if one rather works with the hyperbolic distance}
$d(x,y)=\mid\log x-\log y\mid$.

This variational approach allows to define candidate means of an arbitrary number of matrices on any connected Riemannian manifold   $\mathcal M$. Such manifolds carry the structure of a metric space whose distance function is the arclength of a minimizing path between two points. Indeed the mean of $x_1,\cdots x_N$ on $\mathcal M$, can be defined as the  minimizer of the sum of squares $ \sum_i d(x,x_i)^2$ where $d$ is the geodesic distance on $\mathcal M$, whenever the unique minimizer exists and unique. Such a mean is known as the Riemannian barycenter, of Karcher or Fr\'echet mean. When only two points are involved, the Karcher mean is the midpoint of the minimizing geodesic connecting those two points and it is usually called {the} Riemannian mean. The main advantage of  the Karcher mean is to readily extend any mean that can be defined as a geodesic midpoint, to an arbitrary number  of points. Unfortunately the mean can rarely be given in closed form, and {is typically computed} by an optimization algorithm on the manifold (see e.g. \cite{absil-book} for more information on this branch of optimization).  In \cite{Karcher} it has been shown that the Karcher mean is uniquely defined on manifolds with non-positive sectional curvature everywhere. On arbitrary manifolds with upper bounded sectional curvature, the Karcher mean exists and is unique in geodesic balls with sufficiently small radius \cite{Afsari}.

\subsection{Ando's mean as a Riemannian mean in the cone $\mathrm{P_n}$}

Any positive definite matrix admits  the factorization $A=Y Y^T$, $Y
\in \mathrm{Gl(n)}$, and the factorization is invariant by rotation $Y
\mapsto Y O$. As a consequence,  the cone of positive definite
matrices has a homogeneous representation $\mathrm{Gl(n)} / \mathrm{O(p)}$. The space
is reductive   and thus admits a   $\mathrm{Gl(n)}$-invariant metric
called the natural metric of the cone of positive definite matrices
\cite{faraut}. If $D_1,D_2$ are two tangent vectors at
$A\in\mathrm{P_n}$, the metric is given by $
g^{P_n}_A(D_1,D_2)=\tr{D_1A^{-1}D_2A^{-1}}$. With this definition, a
geodesic (path of shortest length)  at arbitrary $A\in \mathrm{P_n}$
is  \cite{moakher05,smith-2005}:
$\gamma(t)=A^{1/2}\exp(tA^{-1/2}DA^{-1/2})A^{1/2},~ t>0,$ and
the corresponding geodesic distance is
\begin{align*}d_{P_n}(A,B)=
d(A^{-1/2}BA^{-1/2},I)
&=\norm{\log(A^{-1/2}BA^{-1/2})},\\&=\sqrt{\sum_k \log^2(\lambda_k)},\end{align*}where $\lambda_k$ are
the generalized eigenvalues of the pencil $A-\lambda B$, i.e., the roots of
$\det(AB^{-1}-\lambda I)$. The distance is invariant to action by congruence of $\mathrm{Gl(n)}$ and matrix inversion.

The geodesic characterization provides a closed-form expression of the Riemannian
mean of two matrices $A,B\in \mathrm{P_n}$. The geodesic $A(t)$ linking $A$ and $B$ is
$$A(t)=\exp^{P_n}_A(tD)=A^{1/2}\exp(t\log(A^{-1/2}BA^{-1/2}))A^{1/2},$$ where
$D=\log(A^{-1/2}BA^{-1/2})\in \mathrm{Sym(n)}$. The midpoint is obtained
for $t=1/2$: $M(A, B)=A^{1/2}(A^{-1/2}BA^{-1/2})^{1/2}A^{1/2}$  and it corresponds to the definition (\ref{Ando}).

\subsection{Mean of positive definite matrices and Karcher mean in the cone $\mathrm{P_n}$}\label{Grass:subsec}

For $A_1,..A_N\in\mathrm{P_n}$ viewed  as a Riemannian manifold
endowed with the natural metric, the Karcher mean  is defined as a
minimizer of $X\mapsto \sum_1^N d(X,A_i)^2$, i.e. a least squares
solution that we shall denote  $\mathrm{ls}(A_1,\cdots, A_n)$ as in
\cite{bhatia-2006}. The manifold $\mathrm{P_n}$ endowed with the
natural metric is complete, simply connected, and has everywhere a
negative sectional curvature. As a consequence,  the Karcher mean is
uniquely defined. It has been proposed by \cite{moakher06} as a
natural candidate for generalizing the Ando mean to $N$ matrices,
and studied by \cite{bhatia-2006,LL2011,BK2011,holbrook}. It can
mainly be calculated via a simple Newton method on $\mathrm{P_n}$,
or by a stochastic gradient algorithm \cite{Arnaudon}. However,
finding a closed-form expression of the Karcher mean of three or
more matrices of $\mathrm{P_n}$ remains an open question. Several
recent papers address the issue of approximating the Karcher mean
via simple algorithms \cite{ando,petz}.

\subsection{Mean of projectors and Karcher  mean in the Grassmann manifold}\label{Grass:subsec2}

{The Riemmanian approach to the definition of means provides a natural definition for the mean of  $p$-dimensional projectors in $\RR^n$, which forms a subset of $\Sy$:  \begin{align}\label{proj:eq}
\{P\in\RR^{n\times n}/~P^T=P,~P^2=P,~\tr{P}=p\},
\end{align}
This set is in bijection with the Grassmann manifold of
$p$-dimensional subspaces Gr(p,n)  (e.g. \cite{absil-book}). This
manifold can be endowed with a natural Riemannian structure. The
squared distance between two subspaces is merely   the sum of the
squares of the principal angles between those two $p$-planes (see,
e.g. \cite{golub-book} for a definition of principal angles).  The
Riemannian mean of two subspaces is uniquely defined as soon as all
the principal angles between those subspaces are strictly smaller
than $\pi/2$. In other words, the injectivity radius at any point,
i.e. roughly speaking the distance at which the geodesics cease to
be minimizing, is equal to $\pi/2$ on this manifold. The Karcher
mean of $N$ subspaces $S_1,\cdots S_N$ of $\mathrm{Gr}(p,n)$ is
defined as the least squares solution  that minimizes   $  X\mapsto
\sum_1^N d_{\mathrm{Gr(p,n)}} (X,S_i)^2$. The latter function is
equal to $  \sum_{i=1}^N \sum_{j=1}^p \theta_{ij}^2$ where
$\theta_{ij}$ is the $j$-th principal angle between  $X$ and $S_i$.
For $N >2$, there is no closed-form solution for the mean subspace
$X$.  For this reason, the Riemannian mean is often approximated by
the chordal mean, see Section \ref{filtering:sec} and more generally \cite{sarlette}.
As it is well-known  the sectional curvature of the Grassmann
manifold does not exceed 2, and the injectivity radius is $\pi/2$,
we have guarantees that the Karcher mean exists and is uniquely
defined in a geodesic ball of radius less than ${\pi}/{(4\sqrt 2)}$ \cite{Afsari}.

 The Karcher mean  of projectors    in $\mathrm{Gr}(p,n)$ is a natural rank-preserving rotation-invariant mean that is well-defined on a subset of the boundary of the cone. We will use this mean subspace as a basis for the mean of $N$ matrices of $\Sy$. }

\section{A rank-preserving mean of an arbitrary number of matrices of $\Sy$}
\label{sec3}

The proposed extension of the mean from projectors to {arbitrary} matrices of
 $\Sy$ is based on the  decomposition
$$A=UR^2U^T, $$ of any matrix $A\in\Sy$, with $U$ an orthonormal matrix in $\mathrm{St(p,n)}$ and $R^2$ a positive definite matrix in $\mathrm{P_p}$. {This matrix decomposition  emphasizes the geometric interpretation of elements of  $\Sy$  as \emph{flat $p$-dimensional ellipsoids} in $\RR^n$. The flat ellipsoid belongs to a $p$-dimensional subspace spanned by the columns of $U$, which form an orthonormal basis of the subspace, whereas the $p\times p$ positive definite matrix $R^2$ defines  the shape of the ellipsoid in the low rank cone $\mathrm{P_p}$}. For each $1\leq i\leq N$, the matrix decomposition $A_i=U_iR_i^2U_i^T$ is defined up to an orthogonal transformation \begin{align}\label{trans:eq}U_i\mapsto U_iO_i,\qquad R_i^2\mapsto O_i^TR_i^2O_i\, ,\end{align}with $O_i\in\mathrm{O(p)}$ since $$A_i=U_iR_i^2U_i^T=U_iO_i(O_i^TR_i^2O_i)O_i^TU_i.$$ The orthogonal transformations do not affect
the Grassmann Riemannian mean,  but do  affect, in general,  the
mean of the low-rank factors since $M(R_1^2,R_2^2)\neq
M(R_1^2,O^TR_2^2O)$ for arbitrary $O\in\mathrm{O(p)}$, where $M$ denotes the Ando mean. Principal
difficulties for defining a proper geometric mean stem from this
ambiguity.

\subsection{Mean of two matrices $A_1$ and $A_2$}

Let $A_1=U_1R_1^2U_1^T$ and $ A_2=U_2R_2^2U_2^T$ be elements of
$\Sy$. The representatives of the two matrices $(U_i,R_i^2), i=1,2,$
are defined up to an orthogonal transformation $$U_i\mapsto
U_iO,\qquad R_i^2\mapsto O^TR_i^2O.$$All the bases
$U_i\mathrm{O(p)}$ correspond to the same $p$-dimensional subspace
$U_iU_i^T$ (Figure \ref{Fig1}). Note that, this representation of a
$p$-dimensional subspace as the set of bases $U_i \mathrm{O}(p)$ is
at the core of the definition of   the Grassmann manifold
$\mathrm{Gr(p,n)}$  as a quotient manifold
\cite{edelman98}$$\mathrm{Gr(p,n)}\approx\mathrm{St(p,n)}/\mathrm{O(p)}.$$
The equivalence classes $U_i\mathrm{O(p)}$ are called the ``fibers".

 {We will systematically assume that the principal angles between span($U_1$) and span($U_2$) are less than $\pi/2$, which is almost surely true if the subspaces span($U_1$) and span($U_2$) are picked  randomly.  In the case of two matrices, this is sufficient to ensure their Karcher mean in Gr(p,n) is unique. To remove the ambiguity in the definition of a mean of two  matrices of $\Sy$, we propose   to pick two particular representatives  $Y_1=U_1Q_1$ and $Y_2=U_2Q_2$ in the fibers $U_1\mathrm{O(p)}$ and $U_2\mathrm{O(p)}$ by imposing that their distance in $\mathrm{St(p,n)}$ does not exceed the Grassmann distance between the fibers they generate:
 \begin{align}\label{mingrass}d_\mathrm{St(p,n)}(Y_1,Y_2) = d_{\mathrm{Gr(p,n)}} (\mathrm{span}(U_1), \mathrm{span}(U_2))\,,\end{align}
Because the projection from  $\mathrm{St(p,n)}$ to Gr(p,n) is a
Riemannian submersion \cite{absil-book}, and Riemannian submersions
shorten the distances \cite{Gallot-book}, this condition admits the
equivalent formulation: $Y_1=U_1Q_1$ and $Y_2=U_2Q_2$ with
\begin{align}\label{min:eq}(Q_1,Q_2)=\text{argmin}_{(O_1,O_2)\in \mathrm{O(p)}\times \mathrm{O(p)}}d_\mathrm{St(p,n)}(U_1O_1,U_2O_2)\, .
\end{align}
which is illustrated by the  picture of Figure \ref{Fig1}: a
geodesic in the Grasmman manifold admits the representation of a
{\it horizontal} geodesic in $\mathrm{St(p,n)}$, that is, in the
present case, a geodesic whose tangent vector points everywhere to a
direction normal to the fiber.
\begin{figure}
\begin{center}
\includegraphics[scale=0.6]{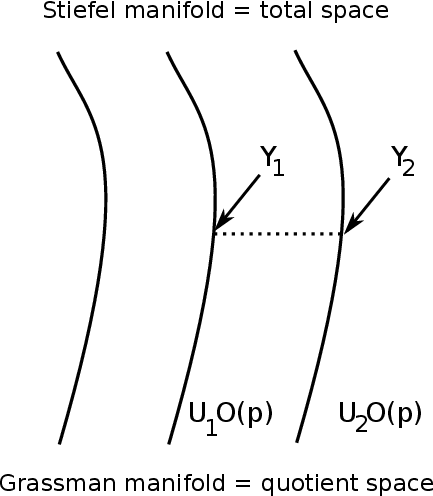}
\caption{$(Y_1,Y_2)$ are two  bases of the subspaces spanned by the
columns of $U_1$ and $U_2$ (the fibers) that minimize the distance
in $\mathrm{St(p,n)}$.  The dashed line represents the shortest path
between those two fibers, thus its horizontal lift (i.e. its
projection) in $ \mathrm{Gr(p,n)}$ viewed as a quotient manifold, is
also a geodesic.} \label{Fig1}
\end{center}
\end{figure}

The following proposition solves the equation \eqref{min:eq}.}

\begin{prop}\label{representatives}
The compact singular value decomposition (SVD) of $U_1^TU_2$ writes
 \begin{equation}
  U_1^T U_2  = O_1 (\cos \Sigma) O_2^T, \label{ou} \end{equation}where
$\Sigma$ is a diagonal matrix whose entries are the principal angles between the $p$ dimensional subspaces spanned by $U_1$ and $U_2$ \cite{golub-book}. If the pair $(O_1,O_2)$ is defined via \eqref{ou}, it is a solution of   \eqref{min:eq}.
\end{prop}

\begin{proof} We use a well-known  result in the Grassmann manifold: the shortest path between two fibers in $\mathrm{St(p,n)}$ coïncides with the geodesic path linking these two fibers in $\mathrm{Gr(p,n)}$,  as the projection on the Grassmann manifold is a Riemannian submersion, and thus shortens the distances (see \cite{Oneill-book,Gallot-book} for results on quotient manifolds). If two bases $Y_1$ and $Y_2$ of the fibers $U_1O(p)$ and $U_2O(p)$ are the endpoints of a geodesic in the Grassmann manifold, they must minimize  \eqref{min:eq}. It is  thus sufficient to prove that $Y_1=U_1O_1$ and $Y_2=U_2O_2$, where $O_1,O_2$ are defined via \eqref{ou}, are the endpoints of a minimizing Grassmann geodesic.

A geodesic in the Grassmann manifold with $Y_1$ as starting point
and $\Delta$ as tangent vector admits the general form
\cite{edelman98}
\begin{align}\label{edel:eq}
\gamma(t)=Y_1V\cos\Theta tV^T+U\sin\Theta t V^T\,,
\end{align} where $U\Theta V^T=\Delta$ is the compact SVD of $\Delta$. We thus propose to consider the following curve
\begin{align*}
Y(t)=Y_1\cos\Sigma t+X\sin\Sigma t\, .
\end{align*}
To define $X$, first assume all principal angles, i.e. all diagonal entries of $\Sigma$, are strictly positive, and let $
X=(Y_2-Y_1\cos\Sigma)(\sin\Sigma)^{-1}$. The curve $Y(t)$ is a geodesic, as it is of the form \eqref{edel:eq} with $\Delta=X\Sigma$ which is a tangent vector as $Y_1^T\Delta=0$ (since $Y_1^TY_2=\cos\Sigma$),  and $X\Sigma$ is a compact SVD as $X^TX=I$. This is because $X^TX=(Y_2^T-\cos\Sigma Y_1^T)(Y_2-Y_1\cos\Sigma)(\sin\Sigma)^{-2}=(I-(\cos\Sigma)^2)(\sin\Sigma)^{-2}=I$ where we used the fact that $Y_2^TY_1=Y_1^TY_2=\cos\Sigma$.  $Y_1$ and $Y_2$ are its endpoints indeed as $Y_2=Y(1)$. If there are null principal angles, it is clear that $Y(t)$ is a geodesic, where $
X=(Y_2-Y_1\cos\Sigma)(\sin\Sigma)^{\dag}$ along the directions corresponding to non-zero principal angles, and where $X$ can be completed arbitrarily with orthonormal vectors along the directions corresponding to null principal angles. Indeed,  along those directions $Y_1$ and $Y_2$, and thus $Y(t)$ coincide, and the value of $X$ does not play any role in the definition of $Y(t)$.

\end{proof}

The following result  allows to understand why the choice of the specific bases $Y_1,Y_2$ is relevant for defining a geometric mean, as explained in the end of this subsection. It proves the rotation of minimal energy (i.e. the closest to identity) mapping the subspace span($A_1$) to span($A_2$) maps $Y_1$ to $Y_2$.

\begin{prop}\label{energy:prop} Let $Y_1= U_1 Q_1$ and $Y_2=U_2 Q_2$ with $(Q_1,Q_2)$ a solution of (\ref{ou}). Then the rotation $R\in\mathrm{SO(n)}$ that maps  the basis $Y_1$ to the basis $Y_2=RY_1$ is a rotation of minimal energy, that is, it minimizes $d_{\mathrm{SO(n)}}(R,I)$ among all rotation matrices that map $Y_1$ to the subspace $\mathrm{span}(U_2)=\mathrm{span}(Y_2)$.
\end{prop}
\begin{proof}
One can assume without loss of  generality that
$Y_1=[e_1,\cdots,e_r]$ where $(e_1,\cdots,e_n)$ is the canonical
basis of $\RR^n$. Moreover, the search space can then be restricted
to the rotations whose $r$ first columns are of the form $Y_2O$,
whereas the $n-r$ remaining columns coincide with the identity
matrix, as the rotation sought must minimize the distance to
identity. Any such rotation mapping $Y_1$ to $Y_2O$ has its first
columns equal to $Y_2O$ and coincides with the identity on the last
$n-r$ columns. Thus we have for any such rotation
$d_{\mathrm{St(n,n)}}(R,I)=d_{\mathrm{St(p,n)}}(Y_2O,I)$. But  as
$\mathrm{SO(n)}=\mathrm{St(n,n)}$ and the  metrics also coincide, we
have $d_{\mathrm{SO(n)}}(R,I)=d_{\mathrm{St(n,n)}}(R,I)$. Thus the
problem boils down to \eqref{min:eq}  and is solved taking $O=I$.
\end{proof}

Having identified some specific representatives  as endpoints of a geodesic in
Gr(p,n), their Riemannian mean in the Stiefel manifold (and in the
Grassmann manifold) is now easily written as the midpoint of the
geodesic:
\begin{align}\label{Grassmean:eq}Y_1\circ Y_2=Y(t=\frac{1}{2})=Y_1\cos\frac{\Sigma}{2}+X\sin\frac{\Sigma}{2}\,.
\end{align}
Note that a weighted mean can be also  computed using  the geodesic
parameterization:
\begin{align}\label{Grassmean:eq2}Y_1\circ Y_2=Y(\alpha)=Y_1\cos(\alpha\Sigma)+X\sin(\alpha\Sigma)\, ,
\end{align}where the weight given to $Y_1$ is $1-\alpha$  and the weight given to $Y_2$ is $\alpha$.

Once $Y_1$ and $Y_2$ have been computed, $R_1$  and $R_2$ are given
by the corresponding representatives\begin{align} R_1^2= Y_1^T A_1
Y_1, \qquad R_2^2= Y_2^T A_2 Y_2. \end{align} The proposed mean of
two matrices $A_1$, $A_2$ is then given by
\[ A_1 \circ A_2= W (R_1^2\#R_2^2) W^T \] where $W$ is  the Riemannian mean of $Y_1$ and $Y_2$ and $R_1^2\#R_2^2$ is the Ando   mean \eqref{Ando} of $R_1^2$ and $R_2^2$ in $\mathrm{P_p}$.

{Proposition \ref{energy:prop} provides a simple geometrical intuition underlying the  definition of the mean: the mean of two flat ellipsoids $A_1$ and $A_2$ is defined in the mean subspace as the geometric mean of two full $p$-dimensional ellipsoids $R_1^2$ and $R_2^2$.   There are several ways to rotate the ellipsoid $A_1$  into the  subspace spanned by $A_2$. Different rotations will yield different respective positions of the two final ellipsoids.  The choice is made univoque and sensible by selecting the minimal rotation. The rotated ellipsoid  then merely  writes $Y_2R_1^2Y_2^T$. Thus $R_1^2$ and $R_2^2$ define the shape of the ellipsoids expressed in the same basis $Y_2$.}   Figure \ref{elli} illustrates the proposed mean of two flat ellipsoids of  $S^+(2,3)$.

\begin{figure}
\begin{center}
\includegraphics[scale=0.35]{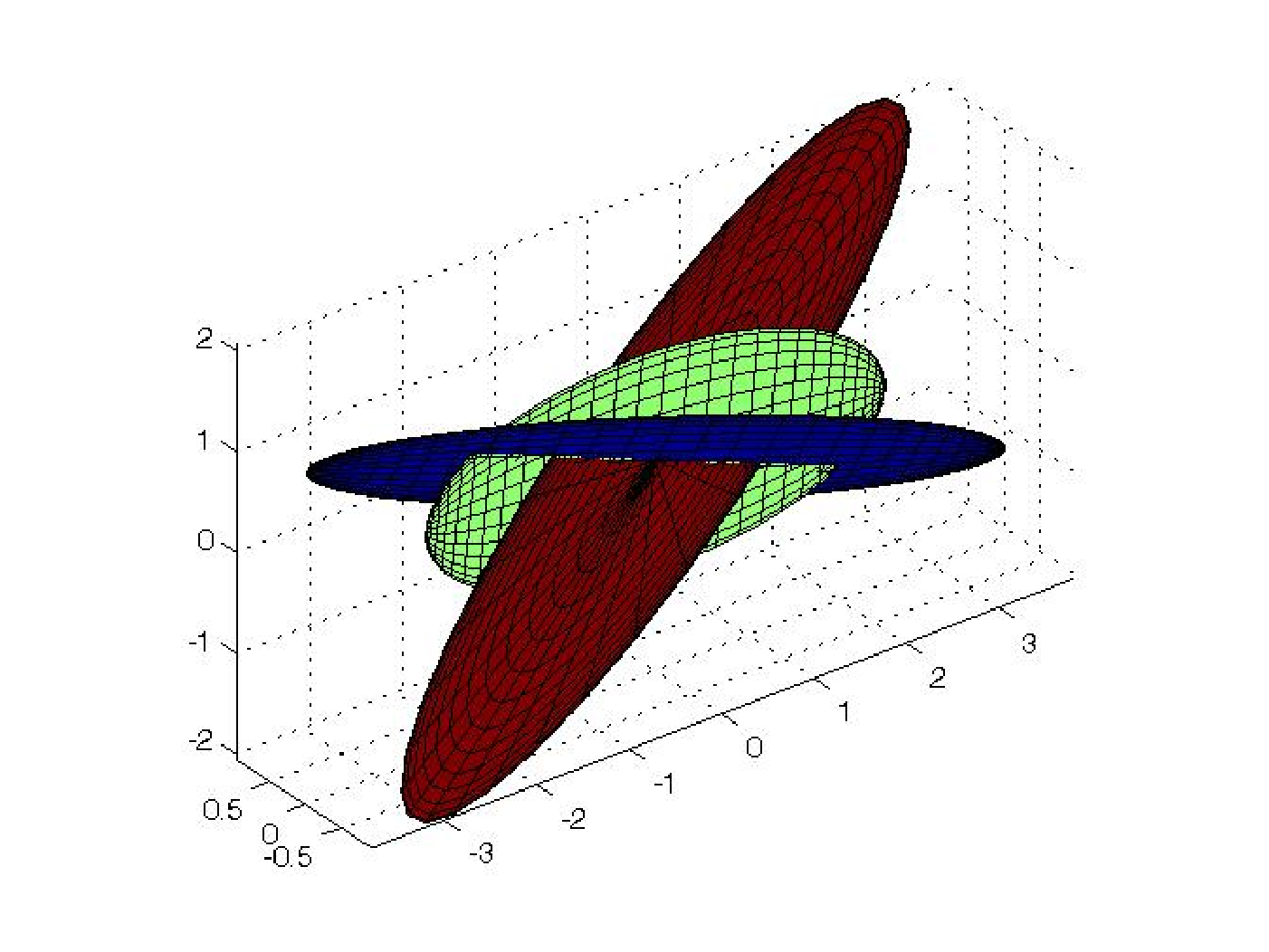}
\caption{Proposed mean in $S^+(2,3)$. The proposed mean is such that both ellipsoids are brought to the mean subspace via a rotation of minimal energy, and then averaged. The resulting mean is a flat ellipsoid, that lives in the mean subspace. }\label{elli}
\end{center}
\end{figure}

\subsection{Generalization to $N$ matrices $A_1, \cdots, A_N\in \Sy$}

 {The construction presented  in the  previous section for two matrices  is now extended to  an arbitrary number of matrices. The main idea is to define a mean $p$-dimensional subspace and to bring   all  flat ellipsoids  $A_1, \cdots, A_N$  to this mean subspace by   a minimal rotation. In the common subspace, the problem  boils  down to compute the geometrical mean of $N$ matrices in $\mathrm{P_p}.$ The construction is achieved through  the following three steps:}

\begin{enumerate}
\item
Let $A_i=U_iR_i^2U_i^T$ for $1\leq i\leq N$.  Suppose that the   subspaces spanned by the columns of the $A_i$'s are enclosed in a geodesic ball of radius less than $\pi/(4\sqrt 2)$ in Gr(p,n). Then define  $W\in$ St(p,n) as an orthonormal basis of the unique Karcher mean of the $U_iU_i^T$'s.

\item
 For each $i$, compute two bases $Y_i$ and $W_i$ of (respectively)  span$(U_i)$ and span$(W)$ such that $d_\mathrm{St(p,n)}(Y_i,W_i) = d_{\mathrm{Gr(p,n)}} (\mathrm{span}(U_i), \mathrm{span}(W))$ i.e. solve problem \eqref{min:eq}.  This is illustrated on Figure \ref{Fig2}. Let $S_i^2=Y_i^TA_iY_i$. The ellipsoid  $A_i$ rotated to the mean subspace  writes
 $
W_iS_i^2W_i^T\, .
$
\begin{figure}
\begin{center}
\includegraphics[scale=0.6]{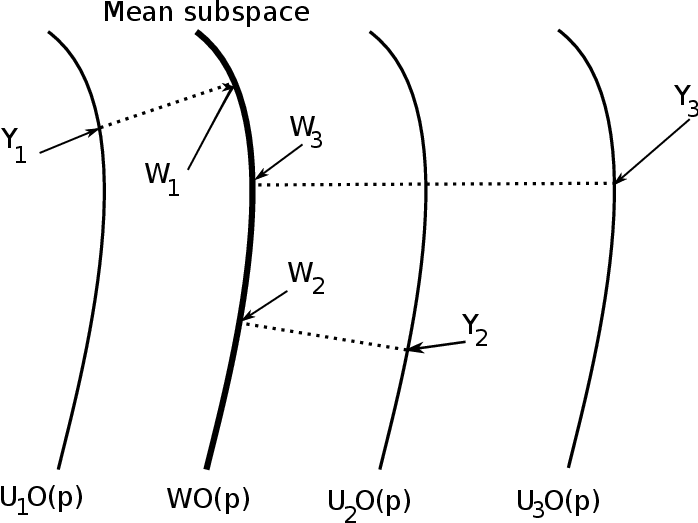}
\caption{The bases $Y_1,\cdots,Y_N$ of the fibers and the bases
$W_1,\cdots,W_N$ of the mean subspace fiber  are such that
$(Y_i,W_i)$ are the endpoints of a geodesic in the Grassmann
manifold.}\label{Fig2}\end{center}
\end{figure}

\item

Let $M$ denote the geometric mean alm or ls on $\mathrm{P_p}$.  For
each $1\leq i\leq N$ let
$T_i^2=W_0^TW_iS_i^2W_i^TW_0\in\mathrm{P_p}$ where
$W_0\in\mathrm{St(p,n)}$ is a fixed basis of the mean subspace.  The
geometric mean of the matrices $A_1,\cdots,A_N$ is defined the
following way:
\begin{align}\label{multimean:def}
A_1\circ...\circ A_N=W_0[M(T_1^2,\cdots,T_N^2)]W_0^T\, .
\end{align}

\end{enumerate}
Those three steps  can summarized as follows: in 1. a
mean subspace is computed, in 2. the ellipsoids are rotated to this
subspace, in 3. they are all expressed in a common basis $W_0$ so
that their geometric mean can be computed in the small dimensional
cone. Note that, although  the  definition \eqref{multimean:def} seems to
depend on the decompositions $A_i=U_iR_i^2U_i^T$ at hand, it will be
proven in the sequel to be invariant to the transformations
\eqref{trans:eq}.  An algorithmic implementation is proposed in Section 6.1.

\section{Properties of the  proposed mean of $N$ matrices of $\Sy$}\label{theoretic:sec}

Throughout this section
\begin{itemize}
\item $M$ will systematically denote one of the geometric means $\mathrm{alm}$ or  $\mathrm{ls}$ on $\mathrm{P_p}$.
\item it will be systematically assumed  the subspaces spanned by the columns of $A_1,...,A_N\in\Sy$ belong to a geodesic ball of radius less than $\pi/(4\sqrt 2)$ in Gr(p,n), so that the Karcher mean of these subspaces is well-defined and unique.
\item with a slight abuse of notation, any projector $UU^T$ where $U\in$St(p,n) will systematically be considered as an element of Gr(p,n), i.e. as a subspace.
\end {itemize}

\subsection{Analytic properties}

\begin{prop}
\label{prop2} On the set of rank $p$ projectors, the mean
\eqref{multimean:def} coincides with the Grassmann Riemannian mean.
On the other hand, when the matrices in $\Sy$ are all supported by the same subspace, \eqref{multimean:def} coincides with the geometric mean
induced by the geometric mean $M$ on the common range subspace of
dimension $p$. More generally \eqref{multimean:def} coincides with
$M$ on the intersection of the ranges.
\end{prop}

\begin{proof}The first two properties are obvious.  The last one is linked to the special choice of a  minimal energy rotation. Indeed, on the intersection of the ranges,  the rotation of  minimal energy  is the identity.\end{proof}

The next proposition proves that the {proposed mean} inherits the several  properties of a geometric mean in the cone. {For the reasons explained in the introduction of the paper}, Properties (PP1) and (PP6-PP7), defined for the mean of three or more matrices in Subsection \ref{multmean:sec}, must be adapted as follows:  (PP1') if $A,B,C$ commute and are supported by the same subspace, then $(A\circ B\circ C)=(ABC)^{1/3}$. (PP6')  For $(\mu,P)\in \RR^*_+ \times \mathrm{O(n)}$ we have $(\mu P^TA\mu P)\circ (\mu P^TB\mu P)\circ (\mu P^TC\mu P)=\mu P^T(A\circ B\circ C ) \mu P$. (PP7')  $(A\circ B\circ C)^{\dag}  =A^{\dag}\circ B^{\dag}\circ C^{\dag}$.

\begin{prop}
\label{prop3} The mean \eqref{multimean:def} with $M=\mathrm{alm}$ or with   $M=\mathrm{ls}$ is well-defined, and  deserves the appellation ``geometric" as it satisfies the properties (PP1'), (PP2-PP5),  and (PP6'-PP7').
\end{prop}

\begin{proof}

(PP1'):  All the $A_i$'s  have the same range. On this common range,
the mean has been proven to coincide with $M$ (Proposition
\ref{prop2}). It thus inherits the (PP1) property. $M$ satisfies
(PP2) and so does \eqref{multimean:def}. To prove permutation
invariance (PP3) it suffices to note that both Grassmann mean and
$M$ are permutation invariant. To prove (PP4), suppose $A_i\leq
A_i^0$ for each $i$. Then $A_i$ and $A_i^0$  have the same range
and  can be written $A_i=U_i R_i^2 U_i^T$ and $A_i^0=U_i R_{i0}^2 U_i^T$. The respective means have the same range, and (PP4)
is then a mere consequence of the monotonicity of $M$. Note that, the monotonicity property in the full rank case
was proved for  $M$=alm in \cite{ando} and  it was first conjectured
for $M$=ls in \cite{bhatia06}, and several proofs were then proposed in
\cite{BK2011,LL2011,holbrook}. Using the same arguments, one can
prove continuity from above of the mean is a consequence of
continuity of $M$. (PP7') can be easily proved noting that for each
$i$ the pseudo-inverse writes $A_i^\dag=U_iR_i^{-2}U_i^T$. Thus the
calculation of the mean of the pseudo-inverse yields the inverse
$T_i^{-1}$ of  $T_i$ and (PP7') is the consequence of
self-duality of $M$.

(PP6'): As for all $\mu>0$ and $i$ we have $\mu A_i=Y_i(\mu
R_i^2)Y_i^T$ invariance with respect to scalings is a mere
consequence of the invariance of $M$. Let $O\in\mathrm{O(n)}$. The
mean subspace in Grassmann of the rotated  ranges of the $(OA_iO^T)$'s is
the rotated mean subspace of the ranges of the $A_i$'s. Proposition
\ref{representatives} says that $Y_i^TW_i=\cos\Sigma$. But for every
$i$ we have $(Y_i^TO^T)(OW_i)=Y_i^TW_i=\cos\Sigma$. Thus the
matrices are transformed according to $W_i\mapsto O W_i$  and the
$T_i's$ are unchanged. The mean of the rotated matrices is thus
$OW_0~M(T_1^2,...,T_N^2)~W_0^TO^T$.\end{proof}

\subsection{The proposed geometric mean as a    Karcher mean in a special case}

In the recent work \cite{bonnabel-sepulchre-simax}, the authors proposed an extension of the affine-invariant metric of the cone to $\Sy$. In this subsection, we explore the links between the Karcher mean in the sense of this metric  and the proposed mean \eqref{multimean:def}. We underline the fact that   the proposed mean is \emph{not} the Karcher mean in the cone.  Yet, we prove that  both means coincide in the meaningful case where all the matrices are rank $p$ projectors.

The metric introduced in  \cite{bonnabel-sepulchre-simax} is as follows. If $(U,R^2)\in \mathrm{St(p,n)}\times \mathrm{P_p}$ represents $A\in\Sy$,  the tangent vectors at $A$ are represented by the infinitesimal variation $(\Delta, D)$, where
\begin{equation}\label{tangent:space}
\begin{aligned}\Delta&=U_\perp B,\quad\quad B\in \RR^{(n-p)\times p},\\D&=RD_0R,\end{aligned}
\end{equation} such that $U_\perp\in \mathrm{St(n,n-p)}$ , $U^TU_\perp=0$, and $D_0\in \mathrm{Sym(p)}$ belongs the tangent space to $\mathrm{P_p}$ at identity. Vectors of the form \eqref{tangent:space} constitute a subset of tangent vectors to the  total space $\mathrm{St(p,n)}\times \mathrm{P_p}$. This subset is called  the horizontal space, and is defined such that each tangent vector of the horizontal space defines a unique tangent vector in the quotient $\Sy$ (i.e. the horizontal space is transverse to the fibers). The chosen metric of $\Sy$ needs only be defined on the horizontal space, and is merely
the weighted sum of the infinitesimal distance in $\mathrm{Gr(p,n)}$ and  in $\mathrm{P_p}$:
\begin{align}\label{metric:def} {g_{k}}_{(U,R^2)}((\Delta_1,D_1),(\Delta_2,D_2))=\tr{\Delta_1^T\Delta_2}+{{k}}~\tr{R^{-1}D_1R^{-2}D_2R^{-1}},~{{k}}>0,
\end{align}generalizing $g^{P_n}$ in a natural way. The space $\Sy\cong (\mathrm{St(p,n)}\times \mathrm{P_p})/\mathrm{O(p)}$ endowed with the metric \eqref{metric:def} is a complete Riemannian manifold, and the metric is invariant to orthogonal transformations, scalings, and pseudo-inversion.

\begin{prop}\label{geo:prop}
Consider $N$ rank $p$ projectors $U_1U_1^T,\cdots, U_NU_N^T\in \Sy$. Then the mean  \eqref{multimean:def} is the Karcher mean of $U_1U_1^T,\cdots, U_NU_N^T$ in the sense of metric \eqref{metric:def}.
\end{prop}
The proof of this proposition is based on two lemmas. Indeed,
this result stems from the fact that  \eqref{multimean:def}  is of the form $WW^T$, where this latter projector is the Karcher mean of the $N$ projectors in the sense of the Gr(p,n) natural metric. This means $WW^T$ is the minimizer of the cost $G(VV^T)=\sum_i d_{\mathrm{Gr(p,n)}}^2(U_iU_i^T,VV^T)$. But the Karcher mean in $\Sy$ is defined as the minimizer of the cost $F(X)=\sum_i d_\Sy^2(U_iU_i^T,X)$. The first following lemma, proves that $F(X)\geq G(\text{span}(X))$ for all $X\in\Sy$. Thus for all $X$ we have $F(X)\geq G(WW^T)$.
 But the second following lemma  proves that  $G(WW^T)=F(WW^T)$. As a result, $WW^T\in\Sy$ minimizes $F$ indeed.

  \begin{lem}\label{lem1}  The distance between arbitrary $A_1,A_2$ in $\Sy$ is lower  bounded by the distance between their ranges in the Grassmann manifold:  $d_\Sy(A_1,A_2)\geq d_{\mathrm{Gr(p,n)}}(\mathrm{span}(A_1),\mathrm{span}(A_2))$   \end{lem}

  \begin{proof}
  A horizontal curve $(U(t),R(t))$ has length  $\int_0^1\sqrt{{g_{k}}_{(U(t),R(t))}(\dot U(t),\dot R(t))}dt$. For two matrices $A_1,A_2\in\Sy$, consider the  horizontal lift $(U(t),R(t))$ of the geodesic linking $A_1$ and $A_2$ in $\Sy$ in the sense of metric \eqref{metric:def}. As the horizontal vector $(\dot U(t),\dot R(t) )$ has a longer norm than the horizontal vector $(\dot U(t),0 )$,   we have $ d_\Sy(A_1,A_2)\geq d_{\mathrm{St(p,n)}}(U(0),U(1))$. Besides, $U(t)$ defines a curve in $\mathrm{St(p,n)}$ linking span($A_1$) and span($A_2$).  As the projection from the Stiefel manifold to the Grassmann manifold viewed as a quotient space $\mathrm{Gr(p,n)}\simeq \mathrm{St(p,n)}/\mathrm{O(p)} $  is a Riemannian submersion, it shortens the distances, i.e.
$d_{\mathrm{Gr(p,n)}}(\text{span}(A_1),\text{span}(A_2))\leq d_{\mathrm{St(p,n)}}(U(0),U(1))$.  This proves the result.
  \end{proof}

  \begin{lem}
  In the particular case where $A_1,A_2$ in $\Sy$ are two projectors, the geodesic joining them in $\Sy$  coincides with the geodesic joining their ranges in $\mathrm{Gr(p,n)}$. It implies $d_\Sy(A_1,A_2)= d_{\mathrm{Gr(p,n)}}(\mathrm{span}(A_1),\mathrm{span}(A_2))$.
  \end{lem}

\begin{proof}
One can find a horizontal curve in $\Sy$ whose length is
$d_{\mathrm{Gr(p,n)}}(A_1,A_2)$, by choosing representatives in
$\mathrm{St(p,n)}$ as in Proposition \ref{representatives}. It is
thus a geodesic in Grassmann, but also in $\Sy$ because of Lemma
\ref{lem1}.
\end{proof}

Beyond the particular  case of projectors, it must be emphasized
that the mean  \eqref{multimean:def} is { not} the Karcher mean   in
the sense of metric \eqref{metric:def}. This is because a horizontal
curve $(U(t),R(t))$ that is made of a geodesic $U(t)$ in $\mathrm{Gr(p,n)}$
and of a geodesic $R(t)$ in $\mathrm{P_p}$ does \emph{not} define a geodesic
in $\mathrm{St(p,n)}$.  For instance, it is possible to construct  a
geodesic joining matrices having the same range, and such that the
mid-point does not have the same range (see
\cite{bonnabel-sepulchre-simax}, Proposition 1). This
counter-example shows Proposition \ref{prop2}, although very
natural, is not satisfied by the Karcher mean, as the mean of
matrices having the same range does not boil down to their geometric
mean within this range. Even if the metric seems natural, and has
been successfully used in several applications (see e.g.
\cite{JMLR,Springer}),  the resulting Karcher mean lacks elementary
expectable properties that the mean \eqref{multimean:def} possesses.


\section{Application to filtering}\label{filtering:sec}

\subsection{Algorithmic implementation and computational cost}

Here we recap the basic steps for  an implementation of the mean.
The calculation of the mean has a numerical complexity   of order
$O(np^2)$. This cost is linear with respect to $n$, a very desirable
feature  in large-scale applications  where $n \gg p$.

\begin{enumerate}
\item For $1\leq i\leq N$ let $U_i$ be  any orthonormal basis of the span of $A_i$.

\item Let $W$ be an orthonormal basis  of the subspace that is the Karcher mean in the Grassmann manifold between the associated subspaces. Instead of minimizing $  \sum_{i=1}^N \sum_{j=1}^p \theta_{ij}^2$, an sensible alternative is to minimize $  \sum_{i=1}^N \sum_{j=1}^p (\sin \theta_{ij})^2$, which corresponds to approximate the angular distance by a chordal distance in $\mathrm{Gr(p,n)}$. Both definitions are asymptotically equivalent for small principal angles. In this case, $W$ can be defined as an orthonormal basis of  the solution subspace, which was shown  in \cite{sarlette} to be the $p$-dimensional dominant subspace of the centroid $\sum_{i=1}^N U_iU_i^T$, and which can easily be found by truncated SVD.

\item For $1\leq i\leq N$
\begin{itemize}
\item The SVD of $U_i^TW$ yields two orthogonal matrices $O_i,O_i^W$ such that $O_i^TU_i^TWO_i^W$ is a diagonal matrix.
\item  Let $Y_i=U_iO_i$ and $W_i=WO_i^W$. Let $S_i^2=Y_i^TA_iY_i$. Let $T_i^2=W^TW_iS_i^2W_i^TW$.
\end{itemize}
\item Compute the geometric mean in the low-rank cone $M(T_1^2,\cdots,T_k^2)$ using methods in the literature \cite{ando, Arnaudon,Bini}.
\item The geometric mean is: $
W~M(T_1^2,\cdots,T_k^2)~W^T.
$
\end{enumerate}

\subsection{Geometric means and filtering applications}

Filtering on $S^+(n,n)$ with the metric \eqref{metric:def} (which is the GL(n)-invariant metric of the cone $\mathrm{P_n}$) was studied extensively for diffusion tensor images (DTI) filtering in \cite{pennec-06,fletcher,arsigny}. It was also applied to signal processing in \cite{smith-2005}, and also seems to be promising in radar processing \cite{barbaresco}. One of the main benefits of this metric is its invariance with respect to scalings which makes it very robust to outliers, i.e. large noise, as the effect of a large eigenvalue is mitigated by the geometric mean. This very  property, which is desirable for means in the interior of the cone, yields a great lack of robustness to (even small) noises as soon as some matrices are rank-deficient. The mean \eqref{multimean:def} inherits the invariance to scalings property, which yields robustness to outliers, without being subject to the same problems in case of rank deficiency, as illustrated by the following proposition.

\begin{prop}
Let $A\in\Sy$, and $R_\epsilon$ be a rotation of magnitude $\epsilon$. If $\mathrm{span}(R_\epsilon A)\cap
\mathrm{span}(A)=\emptyset$, which can be the case with arbitrary small $\epsilon>0$ as soon as $p<n/2$, the Ando mean of $A$ and $ R_\epsilon A$ is the null matrix according to Proposition \ref{pp1}. On the other hand,  $A\circ R_\epsilon A\to A$ when $\epsilon\to 0$.
\end{prop}

This proposition shows that the Ando mean of a stream of noisy measurements $ R_\epsilon A$  of the low rank matrix $A$, is generally the null matrix, even with arbitrarily small noises, whereas it should be close to $A$. On the other hand, it is indeed close to $A$ when the rank-preserving metric proposed in this paper is used. This appears to be a fundamental feature in filtering applications.

\subsection{An application to diffusion tensor images (DTI)}

The tools developed in this paper can be applied to the
processing of diffusion tensor images. These images represent the
diffusion of water in the brain and are considered as representative
of nervous fibres. Each point of the image contains a matrix
belonging to $S^+(3,3)$, but some of them are highly anisotropic. In
this case, a good approximation of the tensors can be defined
by considering on one hand the dominant direction of diffusion, which is an element of $\mathrm{S}^+(1,3)$, and on the other hand the non-dominant flat
ellipsoid which is an element of $\mathrm{S}^+(2,3)$. The smoothing of these images
is performed by an extension of the Perona-Malik algorithm
\cite{peronamalik}, which we decouple into a filtering problem on
the Grassmann manifold of $1$-dimensional subspaces and a filtering
problem on the set $\mathrm{S}^+(2,3)$. Consider a slice image (2D). In the
Grassmann manifold, using the chordal distance, the discrete Perona-Malik
algorithm becomes
\[ \mathbf{V}_{i,j}^{t+1}= \mathbf{V}_{i,j}^t+ \lambda \bigl(c_N \bigtriangledown_N \mathbf{V}+ c_S \bigtriangledown_S \mathbf{V} + c_E \bigtriangledown_E \mathbf{V} + c_W \bigtriangledown_W \mathbf{V}\bigr)|^t_{i,j}\,\]
where $\mathbf{V}$ denotes the principal eigenvector of the tensor under consideration, and $\nabla$ denotes a  difference with the north, south, east, or west nearest neighbor
\begin{eqnarray*}
\bigtriangledown_N \mathbf{V}_{i,j} &=& \mathbf{V}_{i-1,j}- \mathbf{V}_{i,j}, \quad
\bigtriangledown_S \mathbf{V}_{i,j} = \mathbf{V}_{i+1,j}-\mathbf{V}_{i,j},\\
\bigtriangledown_E \mathbf{V}_{i,j} &=& \mathbf{V}_{i,j+1}- \mathbf{V}_{i,j}, \quad
\bigtriangledown_W \mathbf{V}_{i,j} = \mathbf{V}_{i,j-1}-\mathbf{V}_{i,j},
\end{eqnarray*}
and the coefficients are defined by  $c_{N{i,j}}^t= g(|\bigtriangledown_N \mathbf{V}_{i,j}^t|)$, e.g. for the north direction. $g$ is a well-chosen function \cite{peronamalik} that allows to diffuse (and thus regularize) along  the directions of low gradient but not  along the directions of high gradient. This technique  preserves the edges in the image, and thus prevents from blurring the shapes. The dominant eigenvalue is regularized with the usual algorithm for scalar images.


The results of this filtering algorithm, that is here adapted to a
``multiscale" decomposition of the tensors' eigenvalues, are
illustrated in Figure \ref{dtimages} and following figures, where we
can see that the fiber (in red) is well reconstructed by this
method.

\begin{figure}
\begin{center}
\includegraphics[scale=0.6]{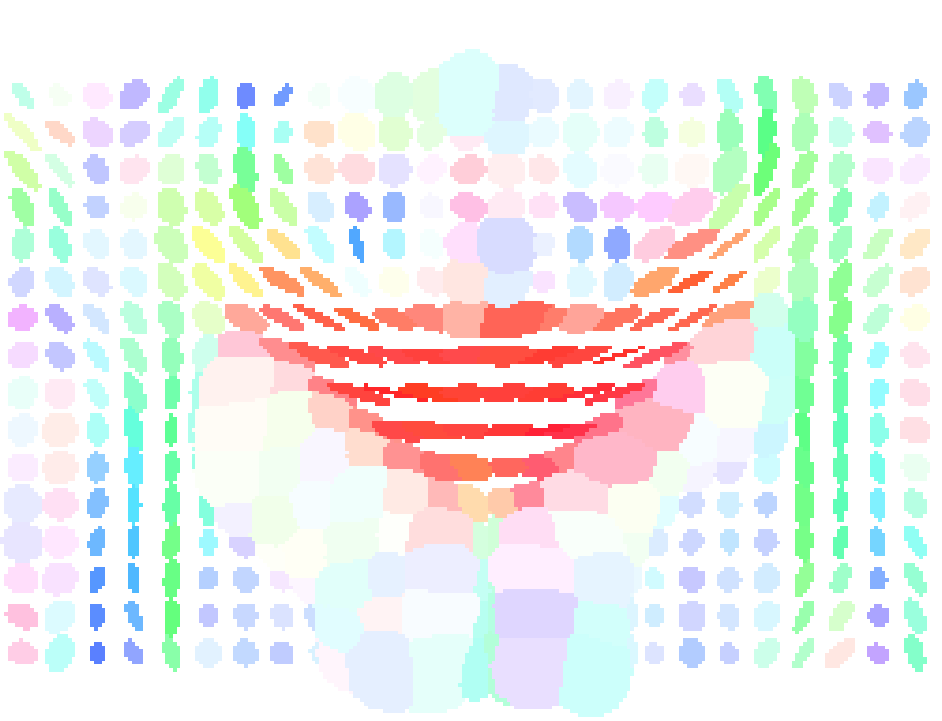}
\caption{Implementation of a Perona-Malik filter on  a Diffusion
Tensor Image. (a): Slice of a Diffusion Tensor Image: zoom on a
fiber of the corpus callosum (in red). Image courtesy of the
Cyclotron Research Centre of the University of
Li\`ege.}\label{dtimages}
\end{center}
\end{figure}

\begin{figure}
\begin{center}
\includegraphics[scale=0.6]{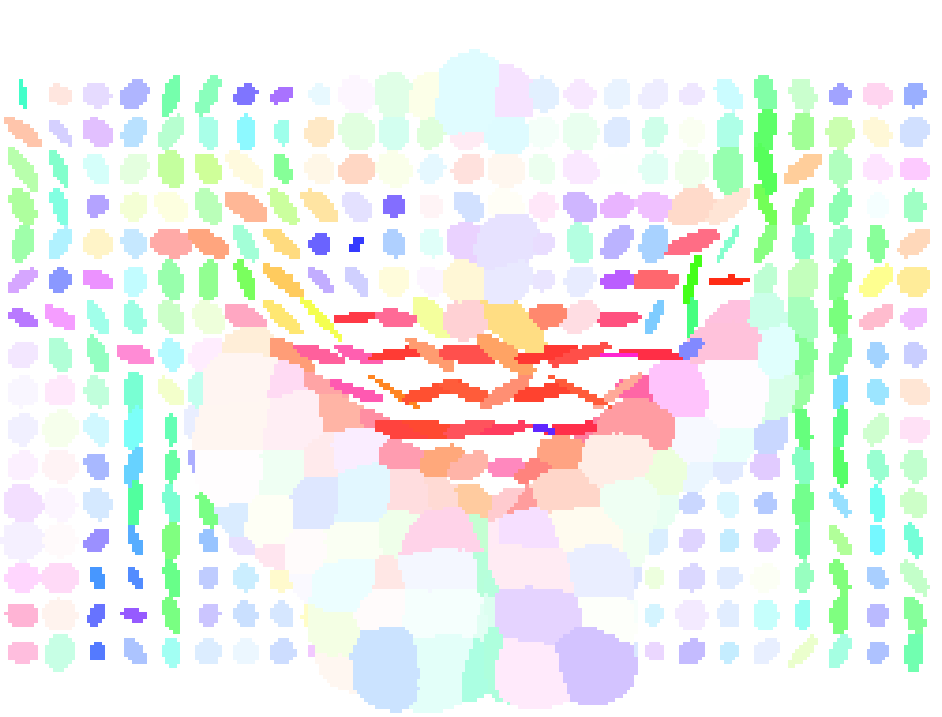}
\caption{Implementation of a Perona-Malik filter on  a Diffusion
Tensor Image. (b): Noise wad added on (a). }
\end{center}
\end{figure}
\begin{figure}
\begin{center}
\includegraphics[scale=0.6]{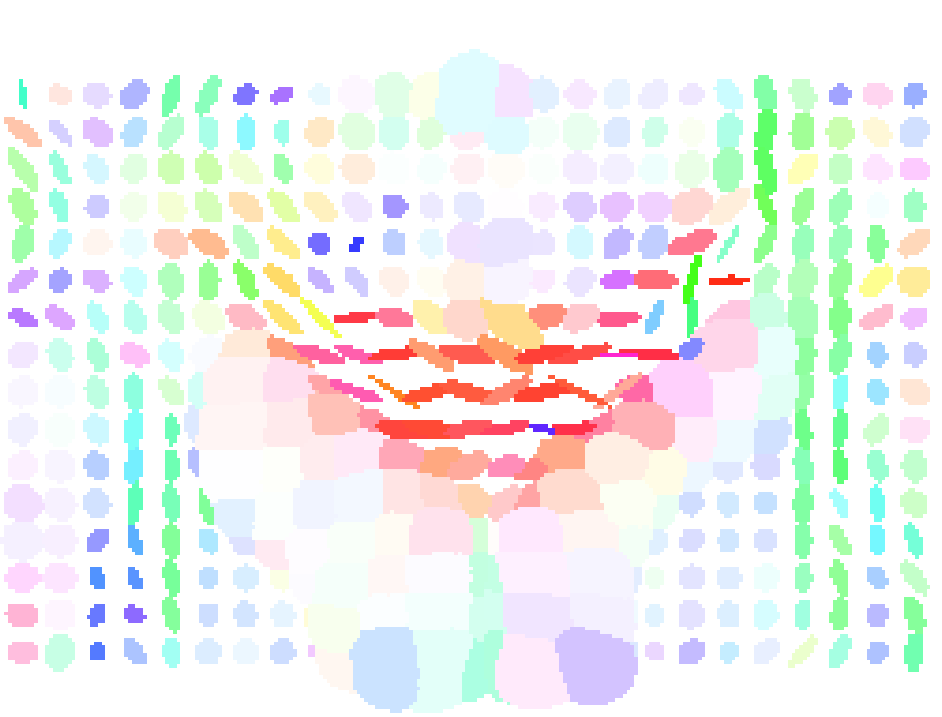}
\caption{Implementation of a Perona-Malik filter on  a Diffusion
Tensor Image. (c): Result of the regularization with the natural
metric on $S^+(3)$. }
\end{center}
\end{figure}
\begin{figure}
\begin{center}
\includegraphics[scale=0.6]{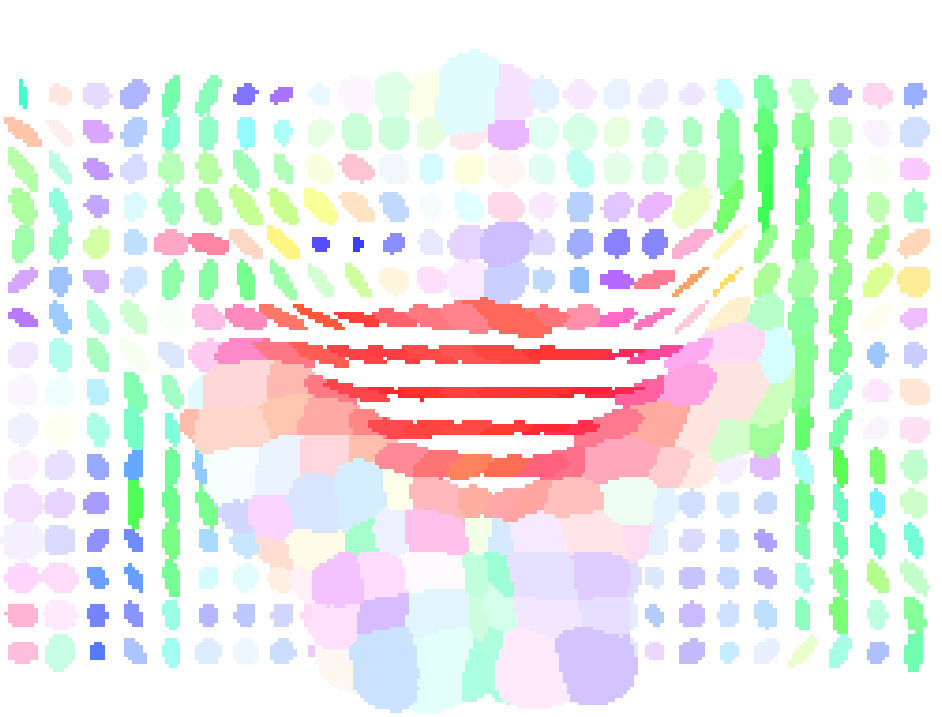}
\caption{Implementation of a Perona-Malik filter on  a Diffusion
Tensor Image. (d): Result of the regularization with an
approximation of the matrices, viewed as a dominant direction of
diffusion coupled with a non dominant flat ellipsoid. The figure
suggests that this regularization performs better than the first one
in the present reconstruction of the fibre.}
\end{center}
\end{figure}

\section*{Acknowledgements}
This paper presents research results of the Belgian Network DYSCO (Dynamical Systems, Control, and Optimization), funded by the Interuniversity Attraction Poles Programme, initiated by the Belgian State, Science Policy Office. The scientific responsibility rests with its authors.


\bibliographystyle{plain}

\begin{thebibliography}{44}

\bibitem{absil-book} 
P.-A. Absil, R.~Mahony, R.~Sepulchre.
\newblock {{Optimization Algorithms on Matrix Manifolds}}.
\newblock {Princeton University Press}, Princeton, NJ, 2007.




\bibitem{Afsari}
B. Afsari,
\newblock{Riemannian Lp center of mass : existence, uniqueness, and convexity.}
\newblock{Proceedings of the American Mathematical Society} 139(2) (2011) 655--673.

\bibitem{ando} 
T.~Ando, C.K. Li, R.~Mathias,
\newblock Geometric means,
\newblock {Linear Algebra Appl.} 385 (2004) 305--334.

\bibitem{ando78} 
T.~Ando,
\newblock Topics on operator inequalities,
\newblock In {Division of Applied Mathematics, Hokkaido University,
Sapporo, Japan}, 1978.

\bibitem{Arnaudon}
M.  Arnaudon, C. Dombry, A. Phan and Le Yang,
\newblock{ Stochastic Algorithms for computing means of probability measures.}
\newblock{Stochastic Processes and their Applications} 122 (2012) 1437--1455.

\bibitem{arsigny} 
V.~Arsigny, P.~Fillard, X.~Pennec, N.~Ayache.
\newblock Geometric mean in a novel vector space structure on symmetric
positive-definite matrices,
\newblock {SIAM J. Matrix Anal. Appl.} 29 (2007) 328--347.

\bibitem{barbaresco}
  F. Barbaresco,
  \newblock{Innovative tools for radar signal processing based on Cartan's geometry of symmetric positive-definite matrices and information geometry.}
  \newblock{IEEE Radar Conference, 2008.}

\bibitem{bhatia06} 
{ R. Bhatia}, {Positive definite matrices},
Princeton University Press, Princeton, NJ, 2006.

\bibitem{bhatia-2006} 
R. ~Bhatia, J.~Holbrook,
\newblock Riemannian geometry and matrix geometric means,
\newblock {Linear Algebra Appl.} 413 (2006) 594--618.

\bibitem{BK2011} 
R. ~Bhatia, R. Karandikar,
\newblock Monotonicity of the geometric mean,
\newblock {Math. Ann.} 353(4) (2012) 1453--1467.

\bibitem{Bini} 
D. Bini, B. Meini and F. Poloni,
\newblock An effective matrix geometric mean satisfying the
Ando-Li-Mathias properties,
\newblock {Math. Comp. } 79 (2010) 437--452.

\bibitem{bonnabel-sepulchre-simax} 
S.~Bonnabel, R.~Sepulchre,
\newblock Riemannian metric and geometric mean for positive semidefinite
matrices of fixed rank,
\newblock {SIAM J. Matrix Anal. Appl.} 31 (2009) 1055--1070.

\bibitem{Bonnabel:MTNS10} 
S.~Bonnabel, R.~Sepulchre,
\newblock Rank-preserving geometric means of positive semi-definite matrices,
\newblock{International Symposium on Mathematical Theory of Networks and Systems (MTNS)} (2010) 209--2014.

\bibitem{Springer} 
S.~Bonnabel, R.~Sepulchre,
\newblock {The geometry of low-rank Kalman filters,}
\newblock in: F. Nielsen and R. Bhatia (Eds.), Matrix Information Geometry, Springer Verlag,
 2012, 53--68.



\bibitem{burbea-rao} 
{ J. Burbea, C. R. Rao}, {Entropy differential metric, distance and
divergence measures in probability spaces: A unified approach},
J. Multivariate Anal. 12 (1982) 575--596.

\bibitem{corach} 
{ G. Corach, H. Porta, and L. Recht}, {Geodesics and operator means in the space of positive operators},
Int. J. Math. 4 (1993) 193--202.


\bibitem{edelman98} 
{A. Edelman, T. A. Arias, S. T. Smith},
\newblock{The geometry of algorithms with orthogonality constraints},
SIAM J. Matrix Anal. Appl. 20 (1998), 303--353.



\bibitem{faraut} 
{J.~Faraut, A.~Koranyi}, {Analysis on Symmetric Cones},
Oxford University Press, New York, 1994.


\bibitem{fletcher} 
{ P. D. Fletcher, S.~Joshi}, {Riemannian geometry for the statistical
analysis of diffusion tensor data}, Signal Processing. 87 (2007) 250--262.


\bibitem{Gallot-book} 
S. Gallot, D. Hulin, J. Lafontaine,
\newblock {Riemannian geometry.}
\newblock {Springer}, Berlin, 2004.

\bibitem{golub-book} 
G.H. Golub, C.~Van Loan,
\newblock {Matrix Computations},
\newblock {John Hopkins University Press}, Baltimore, 1983.

\bibitem{holbrook}
J. Holbrook,
\newblock{No dice: a deterministic approach to the Cartan centroid.}
\newblock{J. Ramanujan Math. Soc., 27(4) (2012) 509-521. }


\bibitem{Karcher}
H. Karcher,
\newblock{Riemannian center of mass and mollifier smoothing.}
\newblock{Comm. Pure Appl. Math., 30 (1977), pp. 509-541. }

\bibitem{LL2011}
J. Lawson and Y. Lim,
\newblock{Monotonic properties of the least squares
mean,}
\newblock{Math. Ann., 351(2) (2011), pp. 267--279. }


\bibitem{lim2}
J. Lawson, H. Lee, Y. Lim,
\newblock{Multi-variable weighted geometric means of positive definite matrices,}
\newblock {Linear Algebra Appl.} 435(2) (2011) 307--322.

\bibitem{JMLR} 
G. Meyer, S.~Bonnabel, R.~Sepulchre,
\newblock {Regression on fixed-rank positive semidefinite matrices: a Riemannian approach,}
\newblock Journal of Machine Learning Reasearch (JMLR). 12(Feb):593-625, 2011.

\bibitem{moakher06} 
{ M. Zerai, M. Moakher},
{The Riemannian geometry of the space of positive-definite matrices and
its application to the regularization of diffusion tensor MRI data},
{J. Math. Imaging Vision}, submitted.

\bibitem{moakher05} 
{M. Moakher}, {A differential geometric approach to the geometric mean
of symmetric positive-definite matrices}, SIAM J. Matrix Anal. Appl. 26 (2005) 735--747.












\bibitem{Oneill-book} 
B.~O'Neill,
\newblock {Semi-Riemaniann geometry},
\newblock Pure and applied mathematics. Academic Press Inc., New York, 1983.



  \bibitem{pennec-06} 
X.~Pennec, P.~Fillard, N.~Ayache,
\newblock A Riemannian framework for tensor computing,
\newblock {International Journal of Computer Vision.} 66 (2006) 41--66.

\bibitem{peronamalik} 
P.~Perona, J.~Malik.
\newblock Scale-space and edge detection using anisotropic diffusion,
\newblock{IEEE Transactions on Pattern Analysis and Machine Intelligence.} 12 (1990) 629--639.


\bibitem{petz} 
D.~Petz, R.~Temesi,
\newblock Means of positive numbers and matrices,
\newblock {SIAM J. matrix anal. appl.} 27 (2005) 712--720.

\bibitem{pusz} 
W.~Pusz, S.~L. Woronowicz,
\newblock Functional calculus for sesquilinear forms and the purification map,
\newblock {Rep. Mathematical Phys.}8 (1975) 159--170.


\bibitem{sarlette} 
{A. Sarlette, R. Sepulchre},
{Consensus optimization on manifolds},
SIAM J. Control and Optimization, 48 (2009) 56-76.


\bibitem{skovgaard} 
{ L. T. Skovgaard},
{A Riemannian geometry of the multivariate normal model},
Scand. J. Statist. 11 (1984) 211--223.



\bibitem{smith-2005} 
S.T. Smith,
\newblock Covariance, subspace, and intrinsic Cramer-Rao bounds,
\newblock {IEEE-Transactions on Signal Processing.} 53(2005) 1610--1629.

\end{thebibliography}


\end{document}